\documentclass[12pt]{amsart}
\usepackage{SSdefn}
\usepackage{enumitem}
\usepackage{tikz-cd}
\usepackage{tikz}
\usepackage{fdsymbol}
\usepackage{mathtools}
\usepackage[font=footnotesize,labelfont=bf]{caption}
\usepackage{graphicx} 
\title[Resolutions of symmetric ideals]{Resolutions of symmetric ideals via stratifications of derived categories}
\author{Karthik Ganapathy}
\address{Department of Mathematics, University of California, San Diego, CA}
\email{\href{mailto:karthg@umich.edu}{karthg@umich.edu}}
\thanks{}
\urladdr{\url{https://public.websites.umich.edu/~karthg/}}
\usetikzlibrary{%
  matrix,%
  calc,%
  arrows%
}

\newtheorem*{customthm*}{Theorem}

\theoremstyle{plain}
\newtheorem{mainthm}{Theorem}

\date{} 
\begin{document}
\begin{abstract}
We propose a method to unify various stability results about symmetric ideals in polynomial rings by stratifying related derived categories. We execute this idea for chains of $\GL_n$-equivariant modules over an infinite field $k$ of positive characteristic. We prove the Le--Nagel--Nguyen--R\"omer conjectures for such sequences and obtain stability patterns in their resolutions as corollaries of our main result, which is a semiorthogonal decomposition for the bounded derived category of $\GL_{\infty}$-equivariant modules over $S = k[x_1, x_2, \ldots, x_n, \ldots]$. Our method relies on finite generation results for certain local cohomology modules. We also outline approaches (i) to investigate Koszul duality for $S$-modules taking the Frobenius homomorphism (of $\GL_{\infty}$) into account, and (ii) to recover and extend Murai's results about free resolutions of symmetric monomial ideals.
\end{abstract}
\maketitle
\section{Introduction}\label{s:intro}
Recent research has highlighted subtle stability properties exhibited by symmetric ideals in polynomial rings. For instance, Le--Nagel--Nguyen--R{\"o}mer conjecture that the regularity \cite{lnnr21reg} and projective dimension \cite{lnnr20pdim} of chains of homogeneous symmetric ideals
\[I_1 \subset I_2 \subset \ldots \subset I_n \subset \ldots\] 
with $I_n \subset k[x_1, x_2, \ldots, x_n]$ are both eventually linear functions of $n$. 
This conjecture is known to hold in special cases \cite{lnnr20pdim, lnnr21reg, rai21reg, ss22pdim, hnt24edge, ln24mono}. Of particular interest to us is the monomial ideal case where Murai \cite{murai20betti} shows that the minimal free resolutions themselves stabilize by proving that all nonzero entries in (a suitable limit) of their Betti tables are concentrated in finitely many ``lines" of integral slope.

Our paper is a preliminary attempt at unveiling the structural reasons behind Murai's result. To that end, we work with $\Fec(S)$-modules over an infinite field $k$ of characteristic $p > 0$, which are compatible sequences $\{ M(k^n) \}_n$ of $\GL_n$-equivariant $k[x_1, x_2, \ldots, x_n]$-modules (see Section~\ref{s:background} for definitions). The class of $\Fec(S)$-modules over $k$ is sufficiently rich, as it captures the subtle phenomena observed by Murai \cite{murai20betti} while remaining tractable. We studied $\Fec(S)$-modules in \cite{ganglp} under the guise of $\GL_{\infty}$-equivariant modules over $S \coloneqq k[x_1, x_2, \ldots, x_n, \ldots]$; a $\Fec(S)$-module $\{M(k^n)\}_n$ corresponds to such an object by taking the limit of $M(k^n)$ as $n \to \infty$ (Proposition~\ref{prop:equiv}). 

Let $\fm^{[q]} = (x_1^q, x_2^q, \ldots, x_n^q, \ldots )$ for prime powers $q$ and $\Mod_S$ be the category of $\GL_{\infty}$-equivariant $S$-modules. Our main result, proved in Section~\ref{s:derived} is:
\begin{mainthm}\label{thm:semiorth}
    We have a semi-orthogonal decomposition of the bounded derived category
    \begin{displaymath}
        \rD^b_{\fgen}(\Mod_S) = \langle  
        \cT_1, \cT_p, \cT_{p^2}, \ldots, \cT_{q}, \ldots,\cT_{\infty}
        \rangle,
    \end{displaymath} 
  where $\cT_{q}$ for a prime power $q$ is the triangulated subcategory generated by the modules $\Sm \otimes L_{\lambda}$ with $L_{\lambda}$ allowed to vary over all irreducible $\GL_{\infty}$-representations, and similarly $\cT_{\infty}$ is the triangulated subcategory generated by the module $S \otimes L_{\lambda}$.
\end{mainthm}
In Section~\ref{s:proofBC}, we use this result to first prove the LNNR conjectures for $\Fec(S)$-modules and then explain how the strata $\cT_q$ in the above semiorthogonal decomposition captures the eventual behaviour of the slope $q-1$ strands in the Betti table of a $\Fec(S)$-module.  In the final section, we also explain a possible approach to generalize Theorem~\ref{thm:semiorth} to $\fM$-equivariant $S$-modules (where $\fM$ is the group of generalized permutation matrices); we expect this to fully recover and extend Murai's results \cite{murai20betti}.

\begin{remark}
    The remarkable aspect of Theorem~\ref{thm:semiorth} is that the decomposition holds at the level of the \textit{bounded derived category}. Indeed, the statement obtained by replacing $\rD^b_{\fgen}$ with $\rD^+$ and $\cT_q$ replaced appropriately is merely a corollary of the Artin--Rees lemma for the ideals $\fm^{[q]}$ \cite[Lemma~3.9]{ganglp} (see also \cite[Proposition~4.14]{ss19gl2}). 
    We emphasize that such semiorthogonal decompositions rarely hold even for the bounded derived categories of ordinary noetherian rings essentially because local cohomology is rarely finitely generated.
\end{remark}

\subsection{LNNR conjectures} 
\begin{mainthm}[LNNR for $\Fec(S)$-modules over $k$]\label{thm:lnnrVecS}\leavevmode

(1) Assume $M$ is a finitely generated $\Fec(S)$-module over $k$. There exists $a \in \bZ$ and a  prime power $q$ such that $\reg(M(k^n)) = (q-1)n + a$ for $n \gg 0$. 
   
(2) Assume further that $M$ is not flat. There exists $a \in \bN$ such that $\pdim(M(k^n)) = n-a$ for $n \gg 0$.
\end{mainthm}
\begin{example}\label{exm:idealI1}
Let $I$ be the ideal $\fm^{[p]}\fm \subset S$. The corresponding $\Fec(S)$-module is the sequence  $\{I(k^n) = \langle x_i^p x_j, 1 \leq i, j \leq n \rangle\}$. The regularity (resp.~projective dimension) of $I(k^n)$ starting with $n=1$ is \[p+1, 2p-1, 3p-2, 4p-3, \ldots\] (resp.~$0, 1, 2, 3, \ldots$), suggesting that $\reg I(k^n)$ and $\pdim (I(k^n))$ stabilizes to $(p-1)n + 1$ and $n-1$ respectively.
\end{example}
Our proof of Theorem~\ref{thm:lnnrVecS} only relies on Proposition~\ref{prop:coarsesemiS} which is a weaker version of Theorem~\ref{thm:semiorth} about decomposing $ \rD^b_{\fgen}(\Mod_S)$ into $ \langle \cT_{\infty}^{\perp}, \cT_{\infty} \rangle$. Let $\Gamma$ be the functor which assigns the maximal torsion submodule to each $\GL_{\infty}$-equivariant $S$-module. The aforementioned proposition is equivalent to proving $\rR\Gamma$ preserves the bounded derived category, which we show in Theorem~\ref{thm:lcfinitenessS}. In Section~\ref{ss:lnnr}, we relate $\rR^i\Gamma(M)$ to the sequence $\{H^i_{\fm(k^n)}(M(k^n))\}$ of usual local cohomology modules (Corollary~\ref{cor:loccoh}); the finite generation result implies that the depth of $M(k^n)$ is eventually constant whence we get part (2) of Theorem~\ref{thm:lnnrVecS} using the Auslander--Buchsbaum formula. The proof for the regularity conjecture is similar albeit involves more bookkeeping.
\begin{remark}
    Sam--Snowden \cite{ss16gl} proved the LNNR conjectures for $\Fec(S)$-modules over fields of characteristic zero; see \cite{ss22pdim} for the $\Fec(S^{{\otimes} n})$ case as well.
    In their setting, the regularity is eventually constant, whereas in characteristic $p$, the regularity is eventually linear of slope $(q-1)$ with $q$ a prime power. We believe making headway on LNNR for $\Fec(S^{\otimes 2})$-modules over $k$ would shed light on proving the conjecture in its most general form.
\end{remark}

\subsection{Slope \texorpdfstring{$(q-1)$}{(q-1)} strands in Betti tables} We finally explain how the strata $\cT_q$ are related to the free resolution of a $\GL$-equivariant $S$-module. In \cite{ganglp}, we showed that the free resolution of a finitely generated module has finitely many ``strands" with slopes of the form $p^r-1$; we restate this in a convenient form here (Theorem~\ref{thm:bettiS}). Theorem~\ref{thm:semiorth} implies that we can associate complexes $\{F_q(M)\}_q$ such that $F_q(M)$ ``captures" the eventual slope $(q-1)$ behaviour in the resolution of $M$; here $F_q$ is just the projection functor $\rD^b_{\fgen}(\Mod_S) \to \cT_q$ for each prime power $q$. 

\begin{mainthm}[Imprecise version]\label{thm:bettiM}
    Let $M$ be a finitely generated $\GL_{\infty}$-equivariant $S$-module.
    For all prime powers $q$, the resolution of $F_q(M)$ has finitely many strands all of slope $(q-1)$, and the slope $(q-1)$ strands of the resolution of $M$ agrees with that of $F_q(M)$.
\end{mainthm}
We make this theorem precise (Theorem~\ref{thm:bettiMr}) in Section~\ref{ss:betti} and swiftly deduce it using Theorem~\ref{thm:semiorth}.

\begin{example}\label{exm:idealI2}
Let $I$ be the ideal from Example~\ref{exm:idealI1}. We let $\bV \coloneqq k^{\infty}$ be the standard representation of $\GL_{\infty}$ and $\bV^{(1)}$ be the ``Frobenius twist" of $\bV$, i.e., the pullback of $\bV$ along the map $\GL_{\infty} \to \GL_{\infty}$ where $(a_{i,j}) \mapsto (a_{i, j}^p)$.

The module $S/I$ is supported on $\fm^{[p]}$ so $F_q(S/I) = 0$ for $q > p$. Using the short exact sequence
\[ 0 \to \fm^{[p]}/I \to S/I \to S/\fm^{[p]} \to 0\]
and the fact that $\fm^{[p]}/I \cong \bV^{(1)}$ with trivial $S$-action, we get that $F_1(S/I) \cong F_1(\bV^{(1)})$ and $F_p(S/I) \cong F_p(S/\fm^{[p]})$.  Using Corollary~\ref{cor:lcvanishingSm}, we further deduce that $F_1(\bV^{(1)}) \cong \bV^{(1)}$ concentrated in cohomological degree $0$ and similarly $F_p(S/\fm^{[p]}) \cong S/\fm^{[p]}$. Figure~\ref{fig:bettitable} illustrates the content of Theorem~\ref{thm:bettiM} for the Betti table of $S/I$ (working in $\overline{\bF_3}[x_1, x_2, x_3, x_4, x_5]$).
\end{example}

For a $\GL_{\infty}$-equivariant $S$-module $M$, if a given strand in the resolution of $M$ has infinitely many nonzero entries, we expect all entries on that strand to be nonzero eventually; Murai proves this for symmetric monomial ideals.  We haven't been able to prove this result using our methods and suspect Koszul duality comes into play; see Section~\ref{ss:koszul} for some precise questions relating $\cT_{q}$ with the derived category of modules over (Frobenius-twisted) exterior algebras. 

\begin{figure}
\scalebox{0.8}{
 \begin{minipage}{\linewidth}
\begin{displaymath}
\begin{matrix}
& 0 & 1 & 2 & 3 & 4 & 5 &  \\
0: &  & . & . & . & . & . & \\
1: & . & . & . & . & . & . & \\
2: & . & . & . & . & . & . & \\
3: & 5 & 25 & 50 & 50 & 25 & 5 & \\
4: & . & . & . & . & . & . & \\
5: & . & . & . & . & . & . & \\
\end{matrix}
\mspace{10mu} + \mspace{10mu}
\begin{matrix}
& 0 & 1 & 2 & 3 & 4 & 5 &  \\
0: & 1 & . & . & . & . & . & \\
1: & . & . & . & . & . & . & \\
2: & . & 5 & . & . & . & . & \\
3: & . & . & . & . & . & . & \\
4: & . & . & 10 & . & . & . & \\
5: & . & . & . & . & . & . & \\
6: & . & . & . & 10 & . & . & \\
7: & . & . & . & . & . & . & \\
8: & . & . & . & . & 5 & . & \\
9: & . & . & . & . & . & . & \\
10: & . & . & . & . & . & 1 &
\end{matrix}
\mspace{10mu} \approx \mspace{10mu}
\begin{matrix}
   & 0 & 1 & 2 & 3 & 4 & 5 &  \\
0: & 1 & . & . & . & . & . & \\
1: & . & . & . & . & . & . & \\
2: & . & . & . & . & . & . & \\
3: & . & 25 & 50 & 50 & 25 & 5 & \\
4: & . & . & 10 & . & . & . & \\
5: & . & . & . & . & . & . & \\
6: & . & . & . & 10 & . & . & \\
7: & . & . & . & . & . & . & \\
8: & . & . & . & . & 5 & . & \\
9: & . & . & . & . & . & . & \\
10: & . & . & . & . & . & 1 
\end{matrix}
\end{displaymath}
\end{minipage}}
\caption{$\beta(F_1(S/I))$ and $\beta(F_p(S/I))$ only have strands of slope $0$ and $p-1$ respectively, which agree with corresponding strands of $\beta(S/I)$ where $I = \fm^{[p]}\fm$ and $p = 3$. Note that $\Tor_0(F_1(S/I))$ cancels $\Tor_1(F_p(S/I))$ but they cannot interact in higher ($\Tor$-)degree due to grading reasons; this observation is crucial in the proof of Theorem~\ref{thm:bettiM}}\label{fig:bettitable}
\end{figure}

\subsection*{Acknowledgements} The author is grateful to Andrew Snowden for helpful discussions and numerous suggestions, Teresa Yu for sharing her ideas about the LNNR conjecture, and Sridhar Venkatesh for comments on an earlier draft. The author was supported in part by NSF grant DMS-2301871.

\section{Preliminaries}\label{s:background}
We fix an infinite field $k$ of positive characteristic $p$ and let $\bV$ be the infinite-dimensional $k$-vector space with basis $\{e_i\}_{i \in \bN}$. The group $\GL$ (or  $\GL_{\infty}$) is the group of automorphisms of $\bV$. We refer the reader to \cite[Section~2]{gan22ext} for details on polynomial representations of $\GL$ relevant to the contents of this paper. We merely note that the irreducible \textit{polynomial} representations of $\GL$ are indexed by partitions of arbitrary size ${\lambda}$ and denoted $L_{\lambda}$. 

\subsection{\texorpdfstring{$\Fec(S)$}{Vec(S)}-modules}
Let $\Fec_k$ be the category of $k$-vector spaces and $\Pol$ be the category of strict polynomial functors over $k$ in the sense of Friedlander--Suslin.

\begin{defn}
A \textit{$\Fec$-algebra} $A$ is a commutative algebra object $A$ in $\Pol$. A \textit{$\Fec(A)$-module} $M$ is a module object for $A$ in $\Pol$.
\end{defn}

Given a $\Fec$-algebra $A$, note that $A(V)$ is automatically an $\bN$-graded vector space induced from the grading on $\Pol$. Similarly, any $\Fec(A)$-module also carries an $\bN$-grading.

\begin{remark}
    Every $\Fec$-algebra is naturally an $\bN$-graded $\FI$-algebra, in the sense of \cite{nr19fioi}. Indeed, we may pull back along the embedding $\FI \to \Fec$ given by $S \mapsto k^S$. Similarly, any $\Fec(A)$-module gives rise to an $\bN$-graded $\FI$-module for the pullback of $A$.
\end{remark}
For a $\Fec$-algebra $A$, the category of $\Fec(A)$-modules is a Grothendieck abelian category. 

\begin{example}[The $\Fec$-algebra $S$]
    The functor $S$ defined by $V \mapsto \Sym(V)$ is naturally a $\Fec$-algebra. An $S$-module is an assignment $V \mapsto M(V)$ where $M(V)$ is a $\Sym(V)$-module, equipped with a rational action of $\GL(V)$ which extends to $\End(V)$, along with some compatibility criteria as $V$ varies.
\end{example}

\begin{proposition}\label{prop:equiv}
Let $\bV = k^\infty$ be the infinite-dimensional $k$-vector space with basis $\{e_i\}_{i \in \bN}$.
\begin{enumerate}
\item The assignment
    \[ A \mapsto A(\bV)\]
    is an equivalence of categories between $\Fec$-algebras and the category of $\GL$-algebras (in the sense of \cite{gan22ext}). 
\item For a $\Fec$-algebra $A$, the assignment
\[ M \mapsto M(\bV)\]
is an equivalence of categories between $\Fec(A)$-modules and the category of $\GL$-equivariant $A(\bV)$-modules.
\end{enumerate}
\end{proposition}
\begin{proof}
The functor $F \mapsto F(\bV)$ is an equivalence of symmetric monoidal categories between $\Pol$ and $\Rep^{\pol}(\GL)$ consisting of polynomial representations of $\GL$. The result follows.
\end{proof}

\begin{defn}
    A $\Fec(A)$-module $M$ is \textit{flat} if $M(V)$ is a free $A(V)$-module for all $V$. An $S(\bV)$-module is $M(\bV)$ is \textit{semi-induced} if $M(\bV)$ has a finite filtration by $A(\bV)$-modules of the form $A(\bV) \otimes L_{\lambda}$.
\end{defn}

\begin{lemma}[cf.~Section 2.4 in \cite{gan22ext}]\label{lem:flat}
   Let $M$ be a finitely generated $\Fec(S)$-module. Then $M$ is flat if and only if $M(\bV)$ is a semi-induced $S(\bV)$-module.
\end{lemma}
\subsection{Recollections}
For most of this paper, we will work on the $\GL$-algebra side, but Proposition~\ref{prop:equiv} guarantees that we do not lose (or gain) any information. 

Let $A$ be any $\GL$-algebra with $A_0 = k$. We abuse notation and write $S$ for the $\GL$-algebra $S(\bV)$, and by $S$-module, we always mean a $\GL$-equivariant $S$-module unless otherwise specified. For a prime power $q$, let $\fm^{[q]}$ be the ideal generated by the $q$-th power of the variables in $S$.

\begin{defn}
Given an $A$-module $M$, a nonzero element $m \in M$ is $\textit{torsion}$ if there exists a nonzero $\GL$-ideal $I \subset A$ such that $Im = 0$.  
The \textit{torsion submodule} or \textit{$0$-th local cohomology module} $\Gamma(M)$ of $M$ is the submodule generated by all torsion elements. A module $M$ is \textit{torsion-free} if $\Gamma(M) = 0$ and \textit{torsion} if $\Gamma(M) = M$.
\end{defn}

For $S$-modules and $\Sm$-modules, there is an easy way to test whether a module is torsion since the $\GL$-spectrum of $S$ is totally ordered. 
\begin{lemma}\label{lem:locannihilator}
\leavevmode 
\begin{enumerate}
        \item A finitely generated $S$-module is torsion if and only if it is annihilated by $\fm^{[q]}$ for sufficiently large prime power $q$.
        \item A finitely generated $\Sm$-module is torsion if and only if it is annihilated by a power of $\fm^{[q/p]}$.
\end{enumerate}
\end{lemma}
\begin{proof}
   The first part follows from \cite[Theorem~A]{ganglp}, and the second part is \cite[Lemma~3.6]{ganglp}.
\end{proof}

We finally collect some of the main structural results about $\GL$-equivariant $S$-modules from \cite{ganglp} relevant to the contents of this paper. Recall that $t_0(M)$ is the generation degree of an $S$-module.

\begin{proposition}\label{prop:masterprop}
\leavevmode
    \begin{enumerate}
       \item For an injective $\GL$-representation $I$, the $S$-module $S \otimes I$ is injective; all torsion-free injectives are of this form.
        \item The injective envelope of a torsion $\Sm$-module in $\Mod_{\Sm}$ is also torsion.
        \item The injective envelope of a torsion $S$-module is also torsion.
        \item Let $M$ be a finitely generated $S$-module. In $\Mod_S^{\gen}$, the object $T(M)$ has a finite injective resolution by finitely generated injectives.
        \item Let $M$ be a finitely generated torsion-free $\Sm$-module. There exists a finitely generated flat $\Sm$-module $F$ containing $M$ such that $t_0(F/M) <  t_0(M)$. 
        \item Let $M$ be a finitely generated torsion-free $S$-module. There exists a finitely generated flat $S$-module $F$ containing $M$ such that $t_0(F/M) < t_0(M)$. 
    \end{enumerate}
\end{proposition}
\begin{proof}
\leavevmode
    (1) is \cite[Theorem~4.25(2)]{ganglp}; (2) is \cite[Proposition~3.9]{ganglp}; (3) is \cite[Proposition~3.9]{ganglp}; (4) is \cite[Theorem~4.25(4)]{ganglp}.
    
    (5) follows by the shift theorem for $\Sm$-modules (\cite[Theorem~E]{ganglp}). Indeed, by the shift theorem, the $\Sm$-module $\mShq^n(M)$ is flat for $n \gg 0$. The natural map $M \to \mShq^n(M)$ given by $m \mapsto (y_1y_2\ldots y_n)^{q/p} m$ is injective since its kernel is $\Sm$-torsion (\cite[Proposition~5.2(a)]{ganglp}) and its cokernel is generated in smaller degree (\cite[Proposition~5.2(a)]{ganglp}).
    
    The proof of (6) is similar to that of (5) but instead we use the shift theorem for $S$-modules (\cite[Theorem~6.4]{ganglp}) and define the map $M \to \Sh_n(M)$ by $m \mapsto y_1^n m$. (It is easy to check that this map satisfies the requisite properties from \cite[Proposition~5.2(a) \& (e)]{ganglp}.)
\end{proof}

\subsection{Technical lemma about local cohomology}
Given an ideal $\fa \subset A$ and $A$-module $M$, let $\Gamma_{\fa}(M)$ be the maximal submodule of $M$ locally annihilated by a power of $\fa$. We let the \textit{$i$-th local cohomology} be the right derived functor $\rR^i\Gamma_{\fa}$. In this subsection, we prove a technical result which allows us to ``spread out" finiteness results about local cohomology from a closed subset to its infinitesimal thickening. The upshot is that we can extend a (specific) semiorthogonal decomposition from the subcategory of $A$-modules scheme-theoretically supported on $\fa$ to the subcategory of modules set-theoretically supported on $\fa$.

Let $\Mod_A^{\tors}$ be the \textit{torsion category} of torsion $A$-modules. The subcategory $\Mod_A^{\tors} \subset \Mod_A$ is a localizing subcategory. We let $\Mod_A^{\gen} \coloneqq \Mod_A/\Mod_A^{\tors}$
be the \textit{generic category}.

Given an ideal $\fa \subset A$, we let $\Mod_A[\fa^{\infty}]$ be the localizing subcategory of modules locally annihilated by a power of $\fa$, and let $\Mod_{A}[\fa]$ be the subcategory of modules annihilated by $\fa$ (this is the same as $\Mod_{A/\fa}$).

\begin{defn}[Section~4 in \cite{ss19gl2}] 
A Serre subcategory $\cB \subset \cA$ satisfies property (Inj) if injective modules in $\cB$ remain injective in $\cA$ 
\end{defn}

\begin{example}
    Let $\fa = \fm^{[q]}$ for some prime power $q$. The Serre subcategory $\Mod_S[\fa^{\infty}] \subset \Mod_S$ satisfies Property (Inj) by Proposition~\ref{prop:masterprop}.
    \end{example}
 
\begin{proposition}\label{prop:restrofsection}
    Let $\fa \subset \fb$ be $\GL$-prime ideals of $A$, and consider the two Serre quotients 
    \[\overline{T} \colon \Mod_{A}[\fa] \to \Mod_A[\fa]/((\Mod_A[\fa])[\fb^{\infty}]) = \colon \cC\]
    and
    \[ T \colon \Mod_A[\fa^{\infty}] \to \Mod_A[\fa^{\infty}]/\Mod_A[\fb^{\infty}] = \colon \cD \] with right adjoints $\overline{S}$ and $S$ respectively. Assume that the Serre subcategories $\Mod_A[\fb^{\infty}] \subset \Mod_A[\fa^{\infty}]$ and $(\Mod_A[\fa])[\fb^{\infty}] \subset \Mod_A[\fa]$ satisfy property (Inj). Denote by $\iota$ the inclusion functor $\Mod_A[\fa] \to \Mod_A[\fa^{\infty}]$ and $\overline{\iota}$ the induced functor $\cC \to \cD$. Then 
    \begin{enumerate}
    \item $\overline{\iota}$ is a fully faithful functor, and 
    \item the derived functor $\rR {S} \circ \overline{\iota}$ is naturally isomorphic to $\iota \circ \rR \overline{S} $
    \end{enumerate}
\end{proposition}
\begin{proof}
The first part, that $\overline{i}$ is a fully faithful functor, is essentially because $A$-module homomorphisms between two $A/\fa$-modules is the same as $A/\fa$-module homomorphisms between them; see proof of \cite[Proposition~6.1]{ss19gl2} for details. For the rest of the proof, we will identify $\Mod_A[\fa]$ and $\cC$ as a subcategory of $\Mod_A[\fa^{\infty}]$ and $\cD$ respectively from which it is clear that $\overline{T}$ is the restriction of $T$ to $\Mod_A[\fa]$, and so $\overline{S}$ is the restriction of $S$ to $\cC$ by uniqueness of right adjoint. 

We now prove the second part. Let $I$ be an injective in $\cC$. We show first that $I$ is ${S}$-acyclic in $\cD$, which is equivalent to vanishing of $\rR \Hom_A(T, S(I))$ for all $T$ in $\Mod_A[\fb^{\infty}]$ by \cite[Proposition~4.7]{ss19gl2}. Using derived adjunction, we obtain an isomorphism 
\[\rR \Hom_A(T,S(I))=\rR \Hom_{A/\fa}(T \stackrel{\rL}{\otimes}_A A/\fa, S(I)).\]
Since $I$ is injective in $\cC$, the $A/\fa$-module $S(I) \cong \overline{S}(I)$ is injective in $\Mod_{A/\fa}$, and so the above isomorphism can be rewritten as $\ext^i_A(T, S(I)) \cong \Hom_{A/\fa}(\Tor_i^A(T, A/\fa), S(I))$ for all $i\geq 0$. The $A$-module $T$ is annihilated by a power of $\fb$, so its derived functors $\Tor_i^A(T, A/\fa)$ also lie in the subcategory $(\Mod_A[\fa])[\fb^{\infty}]$. But $\Hom_{A/\fa}(T', S(I)) = 0$ for all $T'$ in $(\Mod_A[\fa])[\fb^{\infty}]$ since $S(I)$ is evidently saturated with respect to $\Mod_A[\fa][\fb^{\infty}]$.

Given an object $M \in \cC$, we may take an injective resolution $M \to I^{\bullet}$ in $\cC$. Applying $\overline{S}$ and taking homology computes $\rR \overline{S}$, but the same procedure also computes $\rR S$ as $S \cong \overline{S}$ and injective objects in $\cC$ are $S$-acyclic in $\cD$ by the previous paragraph.
\end{proof}
\begin{corollary}\label{cor:devissagelcfiniteness}
    With the assumptions and notations of the previous proposition, if $\rR \overline{S}$ preserves the bounded derived category $\rD^b_{\fgen}(\cC) \to \rD^b_{\fgen}(\Mod_A[\fa])$, then $\rR S$ also maps the bounded derived category $\rD^b_{\fgen}(\cD)$ to $\rD^b_{\fgen}(\Mod_a[\fa^{\infty}])$.
\end{corollary}
\begin{proof} The result follows by the previous proposition using standard d{\'e}vissage arguments since an object in $\cD$ has a finite filtration by objects in $\cC$.
\end{proof}

\section{Proof of Theorem~\ref{thm:semiorth}}\label{s:derived}
Recall that a semi-orthogonal decomposition is the closest non-degenerate analogue of a direct sum decomposition for triangulated categories. More precisely, 
\begin{defn}
Given a triangulated category $\cT$, a \textit{semi-orthogonal decomposition} 
\[\cT = \langle \cT_1, \cT_2, \ldots, \cT_n, \ldots, \rangle\]
is a collection of full triangulated subcategories $\cT_i$ such that the smallest triangulated subcategory generated by the collection is $\cT$, and for all objects $x \in \cT_i$ and $y \in \cT_j$ with $i < j$, we have $\Hom_{\cT}(x, y) = 0$. 
\end{defn}
We caution that the convention in \cite[Section~4]{ss19gl2}, which we follow, is ``opposite" to most definitions in the literature (in that usually one disallows maps $x \to y$ with $x\in \cT_i$ and $y \in \cT_j$ when $i > j$). 
\begin{remark}
    Theorem~\ref{thm:semiorth} is a refinement of \cite[Theorem~B]{ganglp} which states that $\rD^b_{\fgen}(\Mod_S)$ is generated by the $\GL$-twists of $S$ and $\Sm$ with $q$ allowed to vary.
\end{remark}
We prove Theorem~\ref{thm:semiorth} in three steps: first, we obtain a coarse semi-orthogonal decomposition (Proposition~\ref{prop:coarsesemiS}) for $\rD^b_{\fgen}(\Mod_S)$ into $\rD^b_{\fgen}(\Mod_S^{\tors})$ and $\rD^b_{\fgen}(\Mod_S^{\gen})$ using results about torsion-free injective objects in $\Mod_S$ (Proposition~\ref{prop:masterprop}(1) and (4)); second, we prove a similar decomposition (Corollary~\ref{cor:coarsesemiSm}) for $\rD^b_{\fgen}(\Mod_{\Sm})$ using Proposition~\ref{prop:masterprop}(5); finally, we use Corollary~\ref{cor:devissagelcfiniteness} to extend the second part into the requisite decomposition of $\rD^b_{\fgen}(\Mod_S^{\tors})$. 

\subsection{Decomposing \texorpdfstring{$\Mod_S$}{Mod-S} into torsion and flat modules}\label{ss:mods}

Let $\Gamma \colon \Mod_S \to \Mod_S^{\tors}$ be the functor which assigns the torsion submodule; it is right adjoint to the inclusion functor $\Mod_S^{\tors} \to \Mod_S$.

\begin{proposition}\label{prop:semilcvanishing}
   Let $M$ be a finitely generated $S$-module. Then $M$ is semi-induced if and only if $\rR\Gamma(M) = 0$. 
\end{proposition}
\begin{proof}
    For the only if direction, it suffices to show that $\rR\Gamma(M) = 0$ for an induced $S$-module. Consider the induced module $S \otimes L_{\lambda}$. Given an injective resolution $L_{\lambda} \to I_0 \to I_1 \to \ldots \to I_r \to 0$ of $L_{\lambda}$ in $\Rep^{\pol}(\GL)$, we obtain an injective resolution
    \[ 0 \to S \otimes L_{\lambda} \to S \otimes I_0 \to \ldots \to S \otimes I_r \to 0. \]
   We get the zero complex upon applying $\Gamma$ to the above resolution, so $\rR\Gamma(S \otimes L_{\lambda}) = 0$. 

    The proof of the if direction is a fairly standard argument using Proposition~\ref{prop:masterprop}(6) which we reprint here since we use it several times in this paper. Assume the conclusion holds when $t_0(M) < n$ and let $t_0(M) = n$. Since $\rR\Gamma(M) = 0$, the module $M$ is torsion-free so Proposition~\ref{prop:masterprop} applies and we obtain a short exact sequence $0 \to M \to F \to F/M \to 0$ with $F$ finitely generated flat and $t_0(F/M) < t_0(M)$. By the previous paragraph $\rR\Gamma(F) = 0$ so $\rR\Gamma(F/M) = 0$ which implies $F/M$ is flat by induction on $t_0$. Therefore $M$ is also flat being the kernel of a surjective map between two flat modules.
\end{proof}

\begin{theorem}[Finiteness of Local Cohomology for $S$-modules]\label{thm:lcfinitenessS}
   Let $M$ be a finitely generated $S$-module. The $S$-module $\rR^i\Gamma(M)$ is finitely generated for all $i$ and vanishes for sufficiently large $i$. 
\end{theorem}
\begin{proof}
        By Property (Inj), we have $\rR\Gamma(T) = T$ for a torsion $S$-module. By the previous Proposition, we get $\rR\Gamma(F) = 0$ for a flat $S$-module. These two classes generate $\rD^b_{\fgen}(\Mod_S)$ by \cite[Theorem~B]{ganglp} whence we obtain the result.
\end{proof}

We obtain our first semi-orthogonal decomposition of $\rD^b_{\fgen}(Mod_S)$.
\begin{proposition}\label{prop:coarsesemiS}
  We have a semi-orthogonal decomposition
    \begin{displaymath}
        \rD^b_{\fgen}(\Mod_{S}) = \langle  
        \rD^b_{\fgen}(\Mod_{S}^{\tors}),
        \rD^b_{\fgen}(\Mod_{S}^{\gen}) 
        \rangle,
    \end{displaymath} 
    where we identify $\rD^b_{\fgen}(\Mod_{S}^{\gen})$ as a subcategory of $\rD^b_{\fgen}(\Mod_{S})$ using the (derived) right adjoint to the Serre quotient functor $\Mod_{S} \to \Mod_{S}^{\gen}$.
\end{proposition}

\begin{proof}
   This is a corollary of Theorem~\ref{thm:lcfinitenessS}; see \cite[Proposition~4.15]{ss19gl2} for details.
\end{proof}

\subsection{Decomposing \texorpdfstring{$\Mod_{\Sm}$}{Mod-Sm} into torsion and flat modules}\label{ss:modsm}
Fix a prime power $q$ and let $\mGaq$ be functor which sends an $\Sm$-module to its torsion submodule. The functor $\mGaq$ is right adjoint to the inclusion functor $\Mod_{\Sm}^{\tors} \to \Mod_{\Sm}$. 

We recall some facts about the Hasse--Schur derivative $\mShq \colon \Mod_{\Sm} \to \Mod_{\Sm}$ from \cite[Section~5]{ganglp}. Recall we identify $\mShq(M)$ as the subspace of the $S(k \oplus \bV)$-module $M(k \oplus \bV)$ with the new variable $y_1$. We have a natural map $i\colon \id \to \mShq$ where $m \mapsto y_1^{q/p}m$ with $m \in M$; we also let $i[n] \colon \id \to \mShq^n$ be the map with $m \mapsto y_n^{q/p}y_{n-1}^{q/p} \ldots y_2^{q/p}y_1^{q/p}m$ using obvious identifications. The two functors $\mGaq$ and $\mShq$ are closely related:
\begin{lemma}[cf.~Proposition~5.1 in \cite{ganglp}]
    \leavevmode
    \begin{enumerate}
        \item An $\Sm$-module $M$ is torsion-free if and only if $\ker(i_M) = 0$.
        \item The functor $\mShq$ commutes with $\mGaq$.
        \item We have an isomorphism of functors $\mGaq \cong \colim_n \ker(i[n])$.
    \end{enumerate}
\end{lemma}
In fact, $\mShq$ commutes with the derived functors of $\mGaq$. We explain the set up leading up to Proposition~\ref{prop:shcommutesderivedgamma} but the skip its proof since the proof of \cite[Proposition~5.8]{gan22ext} goes through with obvious modifications. 

Let $M$ be an $\Sm$-module, and let $M \to I^\bullet$ and $\mShq(M) \to J^\bullet$ be injective resolution of $M$ and $\mShq(M)$ respectively. The identity map of $\mShq(M)$ induces a map of complexes $\mShq(I^\bullet) \to J^\bullet$ by the lifting property of the injective resolution $J^\bullet$. Applying $\mGaq$ and taking cohomology, we obtain a map $ H^i(\mGaq(\mShq(I^\bullet))) \to H^i(\mGaq(J^\bullet)) = \rR^i\mGaq(\mShq(M))$. We also have natural isomorphisms,
\begin{displaymath}
H^i(\mGaq(\mShq(I^\bullet))) \cong H^i(\mShq(\mGaq(I^\bullet))) \cong \mShq(H^i(\mGaq(I^\bullet))) = \mShq(\rR^i\mGaq(M)),
\end{displaymath}
where for the first isomorphism, we use that $\mShq$ commutes with $\mGaq$, and for the second isomorphism, we use that $\mShq$ commutes with taking cohomology as it is is an exact functor.
Putting all this together, we get a natural map $F_{i,M} \colon \mShq(\rR^i\mGaq(M)) \to \rR^i\mGaq(\mShq(M))$.
\begin{proposition}\label{prop:shcommutesderivedgamma}
The map $F_i$ defined above is an isomorphism for all $i \ge 0$.
\end{proposition}
\begin{proof}
The proof of \cite[Proposition~5.8]{gan22ext} applies mutatis mutandis.
\end{proof}

We can use the above result to obtain another characterization of flat $\Sm$-modules.
\begin{corollary}\label{cor:lcvanishingSm}
    Assume $M$ is a finitely generated $\Sm$-module. Then $M$ is semi-induced if and only if $\rR\mGaq(M) = 0$.
\end{corollary}
\begin{proof}
    The proof of the only if direction is similar to \cite[Proposition~5.11]{gan22ext}; the converse is similar to the proof of the if direction in Proposition~\ref{prop:semilcvanishing}.
\end{proof}

\begin{proposition}[Finiteness of local cohomology for $\Sm$-modules]\label{prop:lcfinitenessSm}
   Let $M$ denote a finitely generated $\Sm$-module. The $\Sm$-module $\rR^i\mGaq(M)$ is finitely generated for all $i$ and vanishes for sufficiently large $i$. 
\end{proposition}
\begin{proof}
This is similar to the proof of Theorem~\ref{thm:lcfinitenessS} (see also  \cite[Theorem~5.17]{gan22ext}).
\end{proof}
\begin{remark}
We emphasize that Proposition~\ref{prop:shcommutesderivedgamma} is independent of the shift theorem for $\Sm$ (see \cite[Remark~5.9]{gan22ext}). However, Proposition~\ref{prop:lcfinitenessSm} uses \cite[Proposition~6.8]{ganglp} which is a corollary of the shift theorem for $\Sm$.
\end{remark}

\begin{corollary}\label{cor:coarsesemiSm}
  Fix a prime power $q$, we have a semi-orthogonal decomposition
     \begin{displaymath}
         \rD^b_{\fgen}(\Mod_{\Sm}) = \langle  
         \rD^b_{\fgen}(\Mod_{\Sm}^{\tors}),
         \rD^b_{\fgen}(\Mod_{\Sm}^{\gen}) 
         \rangle,
     \end{displaymath} 
    where we identify $\rD^b_{\fgen}(\Mod_{\Sm}^{\gen})$ as a subcategory of $\rD^b_{\fgen}(\Mod_{\Sm})$ using the (derived) right adjoint to the Serre quotient functor $\Mod_{\Sm} \to \Mod_{\Sm}^{\gen}$.
\end{corollary}
\begin{proof}
See \cite[Proposition~4.15]{ss19gl2} for details.
\end{proof}

\subsection{Proof of Theorem~\ref{thm:semiorth}}\label{ss:proofA}
We now have all the tools required to prove Theorem~\ref{thm:semiorth}.
\begin{proposition}\label{prop:semiorthtors}
    For each prime power $q$, let $\cT_q \subset \rD^b_{\fgen}(\Mod_S)$ be the triangulated subcategory generated by the modules $\Sm \otimes L_{\lambda}$ with $\lambda$ allowed to vary. We have an infinite semi-orthogonal decomposition
    \[\rD^b_{\fgen}(\Mod_S^{\tors}) = \langle \cT_1, \cT_p, \ldots, \cT_{q}, \ldots \rangle
    \]
\end{proposition}
\begin{proof}
Let $\overline{\cS_q}$ be the right adjoint to $T_q \colon \Mod_{\Sm} \to \Mod_{\Sm}^{\gen}$ and let 
\[\cS_q \colon \Mod_{S}[(\fm^{[q]})^\infty]/\Mod_{S}[(\fm^{[q/p]})^\infty] \to \Mod_{S}[(\fm^{[q]})^\infty]\]
be the section functor. The functor $\rR \overline{\cS_q}$ preserves the bounded derived category by Proposition~\ref{prop:lcfinitenessSm} so we can spread it out to $\rR \cS_q$ by Corollary~\ref{cor:devissagelcfiniteness}. Using the same corollary, the category $\cT_{q}$ can be identified as objects for which $\rR\Gamma_{q/p}$ vanishes inside $\rD^b_{\fgen}(\Mod_S[(\fm^{[q]})^{\infty})$.

Using the results of \cite[Section~4.2]{ss19gl2} and induction on $r$, we get a semiorthogonal decomposition  
    \[\rD^b_{\fgen}(\Mod_S[(\fm^{[p^r]})^{\infty}]) = \langle \cT_1, \cT_p, \ldots, \cT_{p^r} \rangle. \] 
The proposition follows as $\rD^b_{\fgen}(\Mod_S^{\tors}) = \colim_r \rD^b_{\fgen}(\Mod_S[(\fm^{[p^r]})^{\infty}])$.
\end{proof}

\begin{proof}[Proof of Theorem~\ref{thm:semiorth}]
   This is clear by combining Proposition~\ref{prop:coarsesemiS} with Proposition~\ref{prop:semiorthtors}. 
\end{proof}

\section{Proofs of Theorem~\ref{thm:lnnrVecS} and Theorem~\ref{thm:bettiM}}\label{s:proofBC}
\subsection{Le--Nagel--Nguyen--R{\"o}mer Conjectures}\label{ss:lnnr} 
We first relate the local cohomology functors $\rR^i\Gamma$ on $S$-modules to the familiar one from commutative algebra.

Let $\cA$ be the category of sequences $\{M[n]\}$ where $M[n]$ is a $\bZ$-graded module over $k[x_1, x_2, \ldots, x_n]$ for all $n\geq 0$, and let $\cA_{\lf}$ be the Serre subcategory of objects supported in finitely many degrees, i.e., $M$ is in $\cA_{\lf}$ if and only if $M[n] = 0$ for $n \gg 0$. Under our definition, $\cA_{\lf}$ is not a Grothendieck abelian category as it is {not} closed under arbitrary colimits. We have a canonical functor $\Phi \colon \Mod_S \to \cA$ which assigns the sequence $\{M^{\GL_{\infty-n}}\}_n$ to an $S$-module $M$.

We let $\Mod_S^{\gf} \subset \Mod_S$ be the full subcategory of \textit{generically finite length} $S$-modules, i.e., $S$-modules whose image in $\Mod_S^{\gen}$ has finite length; this abelian category includes all torsion $S$-modules (not just finitely generated ones) and has enough injectives by Proposition~\ref{prop:masterprop}(1)\&(2).

Let $H^{0}_{\fm} \colon \Mod_S^{\gf} \to \cA$ be the functor which assigns the sequence 
$\{H^0_{\fm(k^n)}(M(k^n))\}$ to the $S$-module $M$. We similarly define $H^i_{\fm}$ using the usual local cohomology modules $H^i_{(x_1, x_2,\ldots, x_n)}$. Finally, let $\widetilde{H^i_{\fm}} \colon \Mod_S^{\gf} \to \cA/\cA_{\lf}$ be the composition $T \circ H^i_{\fm}$ for all $i \geq 0$. 

\begin{lemma}
    The $i$-th right derived functor of $\widetilde{H^0_{\fm}}$ is naturally isomorphic to $\widetilde{H^i_{\fm}}$. 
\end{lemma}
\begin{proof}
The sequence $\{\widetilde{H^i_{\fm}}\}$ can be endowed with the structure of a $\delta$-functor. It therefore suffices to show that it is effaceable which we do by proving that for all $i > 0$, the functor $\widetilde{H^i_{\fm}}$ vanishes on the indecomposable injectives of $\Mod_S$. By Property (Inj), either $I$ is a torsion $S$-module or torsion-free.
In the first case, the $S(k^n)$-module $I(k^n)$ is itself torsion so $H^i_{(x_1, x_2, \ldots, x_n)}(I) = 0$ for all $i > 0$. In the case when $I$ is torsion-free instead, we have $I \cong S \otimes I_{\lambda}$ for some partition $\lambda$ by Proposition~\ref{prop:masterprop}(1), and so $H^i_{(x_1, x_2, \ldots, x_n)}(I(k^n))$ is nonzero if and only if $i = n$ by standard facts about local cohomology for free modules, and in particular, $H^i_{\fm}(I)$ supported only in degree $i$ and so $\widetilde{H_{\fm}^i}(I) = 0$ as required.
\end{proof}

\begin{proposition}\label{prop:loccoh}
    The functor $T \circ \Phi \circ \rR^i \Gamma \colon \Mod_S^{\gf} \to \cA/\cA_{\lf}$ is naturally isomorphic to $\widetilde{H^i_{\fm}}$.
\end{proposition}
\begin{proof}
The functors are isomorphic when $i=0$, so the result follows by their universality as $\delta$-functors (we note that $\Phi$ and $T$ are exact functors).
\end{proof}

\begin{corollary}\label{cor:loccoh}
Let $M$ be a finitely generated $S$-module and $i \geq 0$. For $n \gg 0$, the local cohomology $H^i_{\fm(k^n)}(M(k^n))$ is isomorphic to $\rR^i\Gamma(M)(k^n)$. 
\end{corollary}
\begin{proof}
   This is clear from Proposition~\ref{prop:loccoh} and the definition of $\cA_{\lf}$. 
\end{proof}

\begin{remark}
Given a finitely generated $S$-module $M$, the above result shows that the sequence $\{H^i_{\fm(k^n)}(M(k^n))\}_n$ ``limits" to $\rR^i\Gamma(M)$, and therefore has the structure of a finitely generated $\GL$-equivariant $S$-module. We stress that the individual local cohomology modules $H^i_{\fm(k^n)}(M(k^n))$ need not even be \textit{polynomial} representation of $\GL_n$ for small values of $n$. For example, the module $H^i_{\fm(k^i)}(S(k^n))$ is the injective envelope of $k$ over $k[x_1, x_2, \ldots, x_i]$ which is only a rational representation of $\GL_i$.
However, the limit of the $i$-th local cohomology is zero since $H^i_{\fm(k^{n})}(M(k^{n})) = 0$ for all $n > i$ (which we have independently confirmed in Proposition~\ref{prop:semilcvanishing}).
\end{remark}

\begin{lemma}
    Assume $M$ is a finitely generated torsion $S$-module. There exists integer $a$ and prime power $q$ such that for $n \gg 0$, we have $\maxdeg(M(k^n)) = (q-1)n + a$.
\end{lemma}
\begin{proof}
   Assume the support of $M$ is $\overline{\fm^{[q]}}$. Then $M$ has a finite filtration by $S/\fm^{[q]}$-modules. So it suffices to show this result for a finitely generated $S/\fm^{[q]}$-module, which follows (by a double induction on $q$ and $t_0(M)$) using Proposition~\ref{prop:masterprop}(5).
\end{proof}

We now prove the LNNR conjectures for $\Fec(S)$-modules.
\begin{proof} [Proof of Theorem~\ref{thm:lnnrVecS}]
We use the equivalence from Proposition~\ref{prop:equiv} and work with the $\GL$-equviariant $S$-module $M(\bV)$ (we abuse notation and also denote this by $M$).

We first prove part (2). The module $M$ is not flat by assumption so $\rR\Gamma(M) \ne 0$ by Lemma~\ref{lem:flat} and Proposition~\ref{prop:semilcvanishing}. Let $a$ be the minimal non-negative integer $i$ such that $\rR^a \Gamma(M) \ne 0$. For $n \gg 0$, we have $H^a_{\fm(k^n)}(M(k^n)) \cong \rR^a\Gamma(M) (k^n) \ne 0$ and $H^i_{\fm(k^n)}(M(k^n)) \cong \rR^i \Gamma(M) (k^n) = 0$ for $i < a$, both by Corollary~\ref{cor:loccoh} and Theorem~\ref{thm:lcfinitenessS}. Using the cohomological characterization of depth for graded modules, we get that the depth of $M(k^n)$ is $a$ for $n \gg 0$ and in turn, that the projective dimension of $M(k^n)$ is $n - a$ by the Auslander--Buchsbaum formula.

We now prove the regularity conjecture. Assume first that $\rR\Gamma(M) = 0$. By Proposition~\ref{prop:semilcvanishing}, we have that $M$ is semi-induced so its regularity is eventually constant (and equal to its generation degree). Now, assume $\rR\Gamma(M) \ne 0$. By Theorem~\ref{thm:lcfinitenessS}, the $S$-module $\rR^i\Gamma(M)$ is finitely generated for all $i \geq 0$ and vanishes for $i \gg 0$. Let $\{i_1, i_2, \ldots, i_t\}$ be the cohomological indices for which $\rR^i\Gamma(M) \ne 0$. By the previous lemma, for all $1 \leq j \leq t$, there exists integers $r_j$ and $a_j$ such that $\maxdeg(\rR^{i_j}\Gamma(M) (k^n)) = (p^{r_j} -1)n + a_j$ for $n \gg 0$. Using the local cohomology charactarization of regularity, and Corollary~\ref{cor:loccoh}, we get that the regularity of $M(k^n)$ is $\max_j \{(p^{r_j} - 1)n + a_j + i_j\}$. It is now easy to see that there exists a $j$ such that for $n\gg0$, this maximum is attained at $(p^{r_j} - 1)n + a_j + i_j$, as required.
\end{proof}

\begin{remark}
 Rohit Nagpal, Andrew Snowden, and Teresa Yu are studying the category of $\fS_{\infty}$-equivariant $S$-modules. 
 We believe our proof provides a pathway to prove the conjecture for {homogeneous} symmetric ideals, provided many of the results about injective objects and local cohomology hold in that context as well.
\end{remark}

\subsection{Linear strands of positive slope}\label{ss:betti}
Classically, the ``slope $0$" linear strands $\cL_i(M) \coloneqq \Tor_j(M, k)_{i+j}$ for a graded module $M$ itself has the structure of a module over the exterior algebra. In characteristic $p$, the classical picture fails to capture the full story about the resolutions of $S$-modules since these modules have infinite regularity. To rectify these issues, we define the \textit{slope $(q-1)$ strand based at $i$} using the formula 
\[\cL_{q, i}(-) \coloneqq \bigoplus_j \Tor_j(-, k)_{qj + i}.\]  
The object $\cL_{q, i}$ is merely a $\GL$-representation.

Let $\Rep_{\lf}^{\pol}(\GL) \subset \Rep^{\pol}(\GL)$ be the Serre subcategory of  $\GL$representations supported in finitely many degrees, and $\Rep^{\gr\pol}(\GL) \coloneqq \Rep^{\pol}(\GL)/\Rep^{\pol}_{\lf}(\GL)$ be the category of \textit{generic representations} of $\GL$.

\begin{defn}
    For an object $M$ in $\rD^b_{\fgen}(\Mod_S)$, define $\fX(M)$ to be the set of tuples $(q, i)$ such that $\cL_{q, i}(M) \ne 0$ in $\Rep^{\gr\pol}(\GL)$. 
\end{defn}

 For a set of tuples $\fB = \{(q_j, i_j)\}_{j \in I}$ and integer $n$, let $\fB[n] \coloneqq \{(q_j, i_j+q_j n)\}_{j \in I}$.
\begin{lemma} \label{lem:strand}
Let $M \to N \to P \to $ be a distinguished triangle in $\rD^b_{\fgen}(\Mod_S)$. We have,
    \begin{enumerate}
        \item $\fX(M[1]) = \fX(M)[1]$, and
        \item $\fX(P) \subset \fX(M)[-1] \cup \fX(N)$.    
    \end{enumerate}
\end{lemma}
\begin{proof}
    The first part is obvious; we obtain the second part from the long exact sequence of $\Tor$ associated to the short exact sequence
    $0 \to N \to P \to M[-1] \to 0$.
\end{proof}

We now restate \cite[Theorem~C]{ganglp} and Theorem~\ref{thm:bettiM} using these helpful definitions. 

\begin{theorem}[cf.~Theorem~C in \cite{ganglp}]\label{thm:bettiS}
   Let $M$ be a finitely generated $S$-module. The set $\fX(M)$ is finite such that for all $(q, i) \in \fX(M)$, we have $q$ is a prime power. Furthermore, the number of nonzero entries not on $\cL_{q, i}(M)$ for $(q, i) \in \fX(M)$ is finite.
\end{theorem}
 \begin{proof}
    The final statement is \cite[Theorem~C]{ganglp}, but also follows from our proof.
     By Lemma~\ref{lem:strand}, it suffices to show that $\fX(M)$ is a finite set of the prescribed form for a set of objects which generate $\rD^b_{\fgen}(\Mod_S)$. By \cite[Theorem~B]{ganglp},  objects of the form $S/\fm^{[p^r]} \otimes L_{\lambda}$ with $r$ varying over $\bN$ and $\lambda$ varying over all partitions, and the flat objects generate $\rD^b_{\fgen}(\Mod_S)$. For a flat module $M$, we have $\fX(M) = \emptyset$ and using the Koszul complex, we get that $\fX(S/\fm^{[q]} \otimes L_{\lambda}) = \{(q, |\lambda|)\}$ for all prime powers $q$ and partitions $\lambda$.  
 \end{proof}    

Let $\cS_q$ and $T_q$ be as defined in the proof of Proposition~\ref{prop:semiorthtors}, and $\Gamma_q$ be the functor which assigns to each module its maximal submodule locally annihilated by a power of $\fm^{[q]}$. Define $F_q \coloneqq \rR\cS_q \circ T_q \circ \rR \Gamma_q$; this is the  be the projection functor $\rD^b_{\fgen}(\Mod_S^{\tors}) \to \cT_q$ (see \cite[Section~4]{ss19gl2}).

\begin{theorem}[Theorem~\ref{thm:bettiM}, precise version]\label{thm:bettiMr}
   Given a prime power $q$, integer $i$, and finitely generated $S$-module $M$, the linear strands $\cL_{q, i}(M)$ and $\cL_{q, i}(F_q(M))$ are isomorphic as generic graded representations of $\GL$.
\end{theorem}
 \begin{proof}[Proof of Theorem~\ref{thm:bettiM}]
    By Lemma~\ref{lem:strand}, for an object $M$ in $\cT_q$, we get that $(q', i) \in \fX(M)$ only if $q' = q$. The interactions (and possible cancellations) between the $\Tor$-groups of two different components $F_q(M)$ and $F_{q'}(M)$ may only occur in finitely many places since their slopes are different. So we get $\cL_{q, i}(M) \cong \cL_{q, i}(F_q(M))$ as generic graded representations for all $i$ as required.
\end{proof}
\begin{example}\label{exm:K}
Let $K$ be the second syzygy of $k$ as a module over $B \coloneqq S/\fm^{[p]}$, so that we have a quasiisomorphism 
\[\begin{tikzcd}
	0 & K & 0 \\
	0 & {B \otimes \bV} & {B } & k & 0
	\arrow[from=1-1, to=1-2]
	\arrow[from=1-2, to=1-3]
	\arrow[hook, from=1-2, to=2-2]
	\arrow[from=2-1, to=2-2]
	\arrow[from=2-2, to=2-3]
	\arrow[from=2-3, to=2-4]
	\arrow[from=2-4, to=2-5]
\end{tikzcd}\]
whence we get $F_1(K) \cong k[2]$ and $F_p(K) \cong 0 \to B \otimes V \to B \to 0$ (with $B \otimes \bV$ in cohomological degree $0$). In Figure~\ref{fig:K}, the horizontal strand $\cL_{0, 2}(K)$ corresponds to $k[2]$, and the strands $\cL_{p, 1}(K)$ and $\cL_{p, p}(K)$ corresponds to $\cL_{p, 1}(B \otimes \bV)$ and $\cL_{p, p}(B[1])$ respectively.
\end{example}
\begin{figure}
\begin{displaymath}
\begin{matrix}
& 0 & 1 & 2 & 3 & 4 & 5 &  \\
2: & 10 & 10 & 5 & 1 & . & . & \\
3: & 5 & 25 & . & . & . & . & \\
4: & . & . & . & . & . & . & \\
5: & . & 10 & 50 & . & . & . & \\
6: & . & . & . & . & . & . & \\
7: & . & . & 10 & 50 & . & . & \\
8: & . & . & . & . & . & . & \\
9: & . & . & . & 5 & 25 & . & \\
10: & . & . & . & . & . & . & \\
10: & . & . & . & . & 1 & 5 &
\end{matrix}
\end{displaymath}
\caption{$\beta(K(k^5))$ from Example~\ref{exm:K} computed using Macaulay2 \cite{M2} when $p=3$}\label{fig:K}
\end{figure}

\section{Postlude}\label{s:postlude}
\subsection{Koszul duality and Frobenius}\label{ss:koszul}
We recall Koszul duality for $\Mod_S$ in characteristic zero from \cite[Section~6]{ss16gl} (see also \cite[Section~7]{ss19gl2}). Given a finitely generated $S$-module $M$, the linear strands $\cL_{0, i}(M) = \bigoplus \Tor_j(M, k)_{i+j}$ are comodules over the infinite exterior coalgebra; the comodule structure is induced from the corresponding structure on the Koszul resolution of $k$. Using contravariant duality on $\Rep^{\pol}(\GL)$, we obtain a (functorial) action of the skew $\GL$-algebra $R \coloneqq \lw(\bV)$ on the dual of $\cL_i(M)$. Sam--Snowden have given explicit generators for $\rD^b_{\fgen}(\Mod_S)$; it is easy to check that the (dual of the) linear strands of these generators form a {finitely generated} $\GL$-equivariant $R$-module, By the functoriality of the $R$-action, we get that $\cL_i(M)$ is a finitely generated $R$-module for all $i \geq 0$ and vanishes for $i \gg 0$.

Theorems~\ref{thm:bettiS} and~\ref{thm:bettiMr} suggests that in characteristic $p$ we should incorporate actions of the Frobenius twists of $R$. Let $R^{(r)} \coloneqq \lw(\bV^{(r)})$ for all $r \geq 1$. 
\begin{problem}
    Define an action of $R^{(r)} $ on $\cL_{p^r,i}$ and prove finiteness results for this action.
\end{problem}
It is not hard to define an action of $R^{(r)}$ on $S/\fm^{[p^r]} \otimes L_{\lambda}$ using the Koszul resolution of $S/\fm^{[p^r]}$; the classes of such modules, with $r$ allowed to vary, generate $\rD^b_{\fgen}(\Mod_S)$ by \cite[Theorem~B]{ganglp}. However, we do not know how to define this action on the strands of an arbitrary $S$-modules. Once such an action is defined, proving finiteness should be straightforward (since $R^{(r)}$-modules are noetherian). By Theorem~\ref{thm:bettiM} it essentially suffices to define the $R^{(r)}$-action for objects in the $\cT_{p^r}$-stratum for all $r \geq 1$.

\begin{remark}
    Eisenbud--Peeva--Schreyer \cite{eps19tor} define an action of $\lw(I/\fm I)$ on $\Tor^A(M, k)$ where $I \subset A$ is an ideal generated by a regular sequence in an (ordinary) commutative noetherian ring $A$ and $M$ is annihilated by $I$. 
    Their method allows us to define a (functorial) action of $R^{(r)}$ on $S/\fm^{[p^r]}$-modules, but we do not know how to extend this action to modules set-theoretically supported on $\fm^{[p^r]}$. 
\end{remark}
We pose a related problem, inspired by Sam--Snowden's result that in characteristic zero, the category of finite length $S$-modules is equivalent to the category of generic $R$-modules.
\begin{problem}
   Let $\cT_{p^r}$ be the triangulated category defined in Theorem~\ref{thm:semiorth}. Relate the triangulated category $\cT_{p^{r}}$ with the derived category $\rD^b_{\fgen}(\Mod_{R^{(r)}})$.
\end{problem}
\subsection{Murai's result}\label{ss:murai}
For this section, we let characteristic of $k$ be abitrary. Murai \cite{murai20betti} proved a result similar to Theorem~\ref{thm:bettiS} for resolutions of symmtric monomial ideal using combinatorial (topological) methods. We outline a possible algebraic approach to obtain his result. 

Given a group $G$ acting on a $k$-algebra $A$ (via $k$-algebra homomorphisms), recall that a $G$-stable ideal $\fp \subset A$ is \textit{$G$-prime} if for all $G$-stable ideals $\fa, \fb$ with $ \fa\fb \subset \fp$, either $\fa \subset \fp$ or $\fb \subset \fp$.

Let $\fM$ be the subgroup of $\GL$ consisting of generalized permutation matrices (or ``monomial matrices"); see \cite{har18mon} for some results on the representation theory of $\fM_n$ over $\bC$. For each $r, n \in \bN$, let $\cI_{n, r}$ be the $\fM$-stable ideal generated by the monomial $(x_1x_2\ldots x_{n+1})^{r+1}$. We skip the proof of this next result. 
\begin{proposition}
The nonzero $\fM$-prime ideals of $S$ are precisely the ideals $\cI_{n, r}$ with $n, r \in \bN$.
\end{proposition}
A representation of $\fM$ is \textit{polynomial} if it occurs as a subquotient of tensor powers of the standard representation $\bV$. We let $\Rep^{\pol}(\fM)$ be the category of \textit{polynomial representations} of $\fM$. The algebra $S$ is a commutative algebra object in $\Rep^{\pol}(\fM)$; let $\Mod_S$ be be the category of module objects for $S$ in $\Rep^{\pol}(\fM)$. For $r \in \bN$, let $\cT_r$ be the triangulated subcategory of $\rD^b_{\fgen}(\Mod_S)$ generated by the $\fM$-twists of $S/{\cI_{n, r}}$ with $n$ allowed to vary, and $\cT_{\infty}$ be the subcategory generated by the $\fM$-twists of $S$. 
\begin{question}
    Do we get a semiorthogonal decomposition 
    \[
   \rD^b_{\fgen}(\Mod_S)  = \langle \cT_0, \cT_1, \cT_2, \ldots, \cT_r, \ldots, \cT_\infty \rangle?
    \]
\end{question}  
If yes, then the component of an $\fM$-equivariant $S$-module $M$ lying in the $\cT_i$ strata would encapsulate the eventual slope $i$ behaviour in $\beta(M)$ similar to Theorem~\ref{thm:bettiMr}. 

The derived functors of $\Gamma \colon \Mod_S \to \Mod_S^{\tors}$ vanishes for the flat $S$-modules $S \otimes V$ with $V$ a polynomial $\fM$-representation; we may either obtain this as a corollary of Theorem~\ref{thm:lcfinitenessS} (combined with the fact that irreducible $\fM$-representations are direct summands of $\GL$-representations), or using the techniques of Section~\ref{ss:lnnr}. The hard part will be proving structural results like \cite[Theorem~C and Theorem~D]{ganglp} for $\fM$-equivariant $S$-modules which has not been seriously pursued. The viewpoint of $\FI_{\fG_m}$-modules \cite{ss19gmaps} with $\fG_m$ the torus may prove useful similar to how $\FI$-modules with varying coefficients help systematize sequences of symmetric ideals \cite{nr19fioi}.
\bibliographystyle{plain}
\bibliography{bibliography}

\end{document}